\documentclass{aims}
\usepackage{amsmath, amssymb, latexsym, enumerate, mathrsfs}
  \usepackage{paralist}
  \usepackage{graphics} 
  \usepackage{epsfig} 
\usepackage{graphicx}  \usepackage{epstopdf}
 \usepackage[colorlinks=true]{hyperref}
\hypersetup{urlcolor=blue, citecolor=red}

\def\currentvolume{x}
 \def\currentissue{}
  \def\currentyear{yyyy}
   \def\currentmonth{}
    \def\ppages{X--XX}
     \def\DOI{}

\usepackage{moreverb}
\usepackage{marvosym}
\usepackage{color,soul}
\definecolor{lightblue}{rgb}{.90,.95,1}
\sethlcolor{yellow}

\newtheorem{theorem}{Theorem}[section]

\newtheorem{proposition}[theorem]{Proposition}

\theoremstyle{definition}

\theoremstyle{remark}
\newtheorem{remark}[theorem]{Remark}

\title[Remarks on the hierarchical control problems] 
      {Remarks on the hierarchical control problems with model uncertainty}

\author[Getachew K. Befekadu and Eduardo L. Pasiliao]{}


 \email{gbefekadu@ufl.edu; pasiliao@eglin.af.mil}

\thanks{This work was supported in part by the Air Force Research Laboratory (AFRL)}

\begin{document}
\maketitle

\centerline{\scshape Getachew K. Befekadu}
\medskip
{\footnotesize
 \centerline{Department of Mechanical and Aerospace Engineering} 
   \centerline{University of Florida - REEF}
   \centerline{1350 N. Poquito Rd, Shalimar, FL 32579, USA}
} 

\medskip
\centerline{\scshape Eduardo L. Pasiliao}
\medskip
{\footnotesize
 \centerline{Munitions Directorate} 
   \centerline{Air Force Research Laboratory}
   \centerline{101 West Eglin Blvd, Eglin AFB, FL 32542, USA}
} 

%

\bigskip

 \centerline{(Communicated by the associate editor name)}

\begin{abstract}
In this paper, we consider a hierarchical control problem with model uncertainty. Specifically, we consider the following objectives that we would like to accomplish. The first one being of a controllability-type that consists of guaranteeing the terminal state to reach a target set starting from an initial condition; while the second one is keeping the state trajectory of the system close to a given reference trajectory over a finite time interval. We introduce the following framework. First, we partition the control subdomain into two disjoint open subdomains, with smooth boundaries, that are compatible with the strategy subspaces of the {\it leader} (which is responsible for the controllability-type criterion) and that of the {\it follower} (which is associated with the second criterion), respectively. Moreover, we account at the optimization stage for model uncertainty by allowing the {\it leader} to choose its control strategy based on a class of alternative models about the system, whereas the {\it follower} makes use of an approximate model about the system. Using the notion of Stackelberg's optimization, we provide conditions on the existence of optimal control strategies for such a hierarchical control problem, under which the {\it follower} is required to respond optimally to the strategy of the {\it leader} so as to achieve the overall objectives. Apart from the issue of modeling and uncertainty, this paper is a companion to our previous work.
\end{abstract}

\section{Introduction}  \label{S1}
Let $\Omega$ be a regular bounded open domain in $\mathbb{R}^{d}$, with smooth boundary $\Gamma$ of class $C^2$. For an open subdomain $U$ of $\Omega$, we consider a distributed control system, governed by the following partial differential equation (PDE) of parabolic type, with a control distributed over $U$, i.e.,
\begin{eqnarray}
\left.\begin{array}{l}
\dfrac{\partial y}{\partial t} + \mathcal{L}_{t,x} y = u\chi_U \quad \text{in} \quad (0, T) \times \Omega  \\
 y (0, x) = 0 \quad \text{on} \quad \Omega  \\
 y (t, x) = 0 \quad \text{for} \quad (t, x) \in \Sigma \triangleq (0, T) \times \Gamma
\end{array}\right\},   \label{Eq1}
\end{eqnarray}
where
\begin{itemize}
\item $u(t, x) \in L^2((0,\,T) \times U)$ is a control function, $\chi_U$ is a characteristic function of the subdomain $U$, and
\item $\mathcal{L}_{t,x}$ is a second-order linear operator given by
\begin{align}
 \mathcal{L}_{t,x} =  \frac{1}{2} \operatorname{tr} \Bigl \{a(t, x) D_x^2\Bigr\} + \mu(t, x) \cdot \triangledown_x,  \label{Eq2}
\end{align}
where $\mu \colon [0, \infty) \times \mathbb{R}^{d} \rightarrow \mathbb{R}^d$ is uniformly Lipschitz, with bounded first derivative; and $a \colon [0, \infty) \times \mathbb{R}^{d} \rightarrow \mathcal{S}_{d}$ is Lipschitz, with the least eigenvalue uniformly bounded away from zero, i.e., 
\begin{align*}
 a(t, x) \succeq \lambda I_{d \times d}, \quad \forall x \in \mathbb{R}^{d}, \quad \forall t \ge 0,
\end{align*}
for some $\lambda > 0$, $\triangledown_x$ and $D_x^2$ stand, respectively, for the gradient and second order derivatives with respect to the variable $x$.
\end{itemize}

In what follows, we assume without loss of generality that the distributed control function $u(t, x)$ is a known function in $L^2((0,\,T) \times U)$ (or it assumes a zero value in $\Omega$).  Note that $\mathcal{L}_{t,x}$ is uniformly elliptic in $\Omega$ for all $t \ge 0$. Then, for a regular bounded open domain $\Omega$ in $\mathbb{R}^{d}$, with smooth boundary $\Gamma$ of class $C^2$, we can represent the solutions $y(t,x)$ of the PDE in \eqref{Eq1} in terms of a stochastic differential equation (see \cite[pp.~144--150] {Fre06} for additional discussions). Later in Sections~\ref{S2} and \ref{S3}, such a stochastic representation will allow us to include at the optimization stage both a single approximate model and a class of alternative models about the system in our problem formulation.

To this end, let $W_t$ (with $W_0 = 0$) denote a $d$-dimensional standard Wiener process, then there exists a stochastic process $x_t$, $0 \le t \le T$, which is an $\mathcal{F}_t$-adapted to the Wiener process $W_t$, with respect to the pair ($\mu(t, x)$, $a(t,x)$) such that
\begin{align}
x_t = x_0 + \int_0^t \mu(s, x_s) ds + \int_0^t \sigma(s, x_s)dW_s,  \label{Eq3} 
\end{align}
where $\sigma(t, x) = a^{\frac{1}{2}}(t, x)$ is Lipschitz and $x_0$ is any point in $\mathbb{R}^{d}$.

In this paper, we assume that the {\it leader} treats the stochastic evolution equation in \eqref{Eq3} (or equivalently \eqref{Eq1}) as an approximation by taking into account a class of alternative models that are (statistically) difficult to distinguish from \eqref{Eq3}. To construct such a perturbed system model, we specifically replace $W_t$ in \eqref{Eq3} by
\begin{align}
 \widehat{W}_t + \int_0^t \sigma(s, x_s) \triangledown_x \log h(s, x_s) d s,  \label{Eq4}
\end{align}
where $\widehat{W}_t$ (with $\widehat{W}_0 = 0$) is a $d$-dimensional standard Wiener process (which is independent to $W_t$) and $h(t, x) \in C_b^{1, 2} \bigl([0, T] \times \mathbb{R}^{d}\bigr)$ is a strictly positive measurable function satisfying
\begin{align}
 \dfrac{\partial \log h(t, x)}{\partial t} & + \frac{1}{2} \operatorname{tr} \Bigl \{ a(t, x) D_x^2  \log h(t, x) \Bigr\}  + \mu(t, x) \cdot \triangledown_x \log h(t, x) \notag \\
  & \quad\quad = -\dfrac{1}{2} a(t, x) \vert \triangledown_x \log h(t, x) \vert^2 \quad \text{in} \quad [0, T) \times \mathbb{R}^{d}. \label{Eq5}
\end{align}

Note that the second-order linear operator corresponding to the perturbed system is given by
\begin{align}
 \mathcal{L}_{t,x}^h & =  \frac{1}{2} \operatorname{tr} \Bigl \{ a(t, x) D_x^2 \Bigr\} + \Bigl(\mu(t, x) + a(t, x) \triangledown_x \log h(t, x) \Bigr) \cdot \triangledown_x \notag \\
   &\equiv \mathcal{L}_{t,x} + a(t, x) \triangledown_x \log h(t, x) \cdot \triangledown_x. \label{Eq6}
\end{align}
Moreover, the stochastic process $x_t^h$ for $0 \le t \le T$, which is associated with the perturbed system, satisfies the following stochastic representation
\begin{align}
x_t^h = x_0 + \int_0^t \bigl(\mu(s, x_s^h) + a(s, x_s^h) \triangledown_x \log h(s, x_s^h) \bigr) ds + \int_0^t \sigma(s, x_s^h)d\widehat{W}_s,  \label{Eq7} 
\end{align}
where the drift perturbation $\bigl\{a(t, x) \triangledown_x \log h(t, x) \bigr\}_{0 \le t \le T}$ is used as device to parametrize the class of alternative models about the system in \eqref{Eq3}. Here, we use such a class of alternative models in our problem formulation, where the {\it leader} uses this class of alternative models for computing its optimal control strategies and whereas the {\it follower} uses a single approximate model about the system (see also Remark~\ref{R3} in Subsection~\ref{S3.2}).

\begin{remark}  \label{R1}
Note that the function $h(t, x) \in C_b^{1, 2} \bigl([0, T] \times \mathbb{R}^{d}\bigr)$ satisfies ${\partial} h(t, x)/{\partial t} + \mathcal{L}_{t, x} h(t, x) = 0$ in $[0, T) \times \mathbb{R}^{d}$ (i.e., it is the kernel of the operator $(\partial/\partial t + \mathcal{L}_{t,x})$). Moreover, in a different context, such a function also defines, in the sense of Doob, an {\it $h$-path process} and further the process $h(t, x_t)/h(0, x_0)$, $0 \le t < T$, is martingale with respect to the natural filtration $\mathcal{F}_t$ (e.g., see \cite{Jam75} for additional discussions).
\end{remark}

\begin{remark}  \label{R2}
Here, we remark that the measurable function $\sigma(t, x) \triangledown_x \log h(t, x)$ parametrizes absolutely continuous changes of measure to the stochastic representation in \eqref{Eq3} (e.g., see \cite{Gir60} for additional discussions on transforming stochastic processes). Such changes of measure can be used to specify model uncertainty in terms of relative entropy as a single constraint on the entire path of perturbation (e.g., see \cite{PetJD00} or \cite{DaiMJ96} for additional discussions).  
\end{remark}

Before concluding this section it is worth mentioning that some studies on the controllability of systems, governed by parabolic equations, have been reported in literature (to mention a few, e.g., see \cite{Lio94} in the context of Stackelberg optimization; and \cite{AraFS15} and \cite{GuiMR13} in the context of Stackelberg-Nash controllability-type problem). Recently, apart from the issue of modeling and uncertainty, the authors in \cite{BefP15} (which is a companion to the present work) have also provided some results pertaining to a chain of distributed control systems. Note that the rationale behind our framework follows in some sense the settings of these papers, where we provide a mathematical framework that addresses the problem of optimal distributed control, in the context of hierarchical control argument, with model uncertainty.

The remainder of this paper is organized as follows. In Section~\ref{S2}, using the basic remarks made in Section~\ref{S1}, we state the hierarchal control problem considered in this paper. Section~\ref{S3} presents our main results -- where we introduce a hierarchical optimization framework, which takes into account both a single approximate model and a class of alternative models about the system, under which the {\it follower} is required to respond optimally to the strategy of the {\it leader} so as to achieve the overall objectives. Moreover, this section also contains results on the controllability-type problem for such a hierarchal control problem.

\section{Problem Formulation} \label{S2}
We consider a hierarchical control problem, which takes into account model uncertainty at the optimization stage, with the following objectives. The first one being of a controllability-type that consists of guaranteeing the terminal state to reach a target set starting from an initial condition; while the second one is keeping the state trajectory of the overall system close to a given reference trajectory over a finite time interval. Such a problem can be stated as follow:

{\bf Problem~(P)}: Find an optimal control strategy $u^{\ast}(t, x) \in L^2((0,\,T) \times U)$ (which is distributed over $U$) such that
\begin{enumerate} [(i)]
\item {\bf The first objective}: Suppose that we are given a target point $y^{t_{g}}$ in $L^2(\Omega)$.
\begin{itemize} []
\item Then, we would like to have 
\begin{align}
 y(T; u^{\ast}) \in y^{t_{g}} + \alpha B, \quad \alpha > 0, \label{Eq8}
\end{align}
where $y(t; u^{\ast})$ denotes the function $x \mapsto y(t, x; u^{\ast})$, $B$ is a unit ball in $L^2(\Omega)$ and $\alpha$ is an arbitrary small positive number; and, at the same time, by taking into account the class of alternative models about the system (i.e., the representation in \eqref{Eq7}).\footnote{Notice that the condition on the terminal state, i.e., $y(T; u^{\ast}) \in y^{t_{g}} + \alpha B$, is associated with a controllability-type problem with respect to an initial condition $y(0, x) = 0$ on $\Omega$ (e.g., see \cite{Lio88} for additional discussions).}
\end{itemize}
\item {\bf The second objective}: Suppose that we are given a reference trajectory $y^{r_{f}}(t,x)$ in $L^2((0,\,T) \times \Omega)$.
\begin{itemize} []
\item Then, we would like to have the state  trajectory $y(t,x; u^{\ast})$ not too far from the reference $y^{r_{f}}(t,x)$ for all $t \in (0, T)$, while taking into account a single approximate model about the system in \eqref{Eq3} (or equivalently the representation in \eqref{Eq1}).
\end{itemize}
\end{enumerate}

In order to make the above problem (i.e.,  Problem~(P)) mathematically precise, we consider the following hierarchical cost functionals:
\begin{align}
 J_1(u) & = \dfrac{1}{2}{\int\int}_{(0,\,T) \times U}  u^2 dx dt \notag \\
            & \quad \text{ s.t.} \quad  y(T; u) \in y^{t_{g}} + \alpha B, \quad \alpha > 0 \label{Eq9}
\end{align}
and
\begin{align}
 J_2(u) =  \dfrac{1}{2}{\int\int}_{(0,\,T) \times \Omega} & \bigl(y(t; u) - y^{r_{f}}(t,x) \bigr)^2 dx dt \notag \\
            & \quad \quad + \dfrac{\beta}{2}{\int\int}_{(0,\,T) \times U} u^2 dx dt, \quad \beta > 0. \label{Eq10}
\end{align}

Note that, in general, finding such an optimal strategy $u^{\ast} \in L^2((0,\,T) \times U)$ (i.e., the {\it Pareto-optimal solution}) that minimizes simultaneously the above cost functionals in \eqref{Eq9} and  \eqref{Eq10} is not an easy problem. However, in what follows, we introduce the notion of Stackelberg's optimization \cite{VonSta34}, where we specifically partition the control subdomain $U$ into two open subdomains $U_1$ and $U_2$ (with $U_1 \cap U_2 = \varnothing$; and with smooth boundaries $\partial U_1$ and $\partial U_2$ of class $C^2$) that are compatible with the strategy subspaces of the {\it leader} and that of the {\it follower}, respectively. That is,
\begin{align}
 U = U_1 \cup U_2 \,\, \text{up to a set of measurable} \,\, U, \label{Eq11}
\end{align}
where the strategy for the {\it leader} (i.e., $u_1$) is from the subspace $L^2((0,\,T) \times U_1)$ and the strategy for the {\it follower} (i.e., $u_2$) is from the subspace $L^2((0,\,T) \times U_2)$.

Note that if $\chi_{U_i}$, for $i = 1, 2$, denotes the characteristic function for $U_i$ and $u_i$ is the restriction of the distributed control $u$ to $L^2((0,\,T) \times U_i)$. Then, the PDE in \eqref{Eq1} can be rewritten as  
\begin{eqnarray}
\left.\begin{array}{l}
\dfrac{\partial y}{\partial t} + \mathcal{L}_{t,x} y = u_1\chi_{U_1} + u_2\chi_{U_2} \quad \text{in} \quad (0, T) \times \Omega  \\
 y(0, x) = 0 \quad \text{on} \quad \Omega  \\
 y(t, x) = 0 \quad \text{for} \quad (t, x) \in \Sigma 
\end{array}\right\},   \label{Eq12}
\end{eqnarray}
where $y(t, x; u) = y(t, x; (u_1, u_2))$ and $u_i \in L^2((0,\,T) \times U_i)$ for $i=1, 2$.

Here, we assume that the {\it follower} uses the above PDE in \eqref{Eq12} (i.e., the approximate model about the system) for computing its optimal control strategies. Suppose that the strategy for the {\it leader} $u_1 \in L^2((0,\,T) \times U_1)$ is given. Then, the problem of finding an optimal strategy for the {\it follower}, i.e., $u_2^{\ast} \in L^2((0,\,T) \times U_2)$, which minimizes the cost functional $J_2$ is then reduced to finding an optimal solution for
\begin{align}
\inf_{u_2 \in L^2((0,\,T) \times U_2)}  J_2(u_1, u_2) \label{Eq13}
\end{align}
such that 
\begin{align}
 u_2^{\ast} = \digamma(u_1) \label{Eq14}
\end{align}
for some unique mapping $\digamma \colon L^2((0,\,T) \times U_1) \rightarrow L^2((0,\,T) \times U_2)$. Note that if we substitute $u_2^{\ast} = \digamma(u_1)$ into \eqref{Eq12}, then the solution $y(t, x; (u_1, \digamma(u_1)))$ depends uniformly on $u_1 \in L^2((0,\,T) \times U_1)$. Moreover, the controllability-type problem in \eqref{Eq9} is then reduced to finding an optimal solution, by taking into account the class of alternative models of \eqref{Eq7} (or the alternative operator representation in \eqref{Eq6}), for the following optimization problem
\begin{align}
& \inf_{u_1 \in L^2((0,\,T) \times U_1)}  J_1(u_1) \notag \\
& \quad \text{ s.t.} \quad  y(T; (u_1, \digamma(u_1))) \in y^{t_{g}} + \alpha B. \label{Eq15}
\end{align}

In the following section, we provide a hierarchical optimization framework for solving the above problems (i.e., the optimization problems in \eqref{Eq13}, together with \eqref{Eq14} and \eqref{Eq15}), where such a framework allows us to provide conditions on the existence of optimal control strategies for such optimization problems. Note that, for a given $u_1 \in L^2((0,\,T) \times U_1)$, the optimization problem in \eqref{Eq13} has a unique solution on $L^2((0,\,T) \times U_2)$ (see Proposition~\ref{P1}). Moreover, as we will see later on (particularly in Propositions~\ref{P2} and \ref{P3}), the optimization problem in \eqref{Eq15} makes sense if $y(T; (u_1^{\ast}, \digamma(u_1^{\ast})))$ spans a dense subset of $L^2(\Omega)$, when $u_1^{\ast}$ spans the subspace $L^2((0,\,T) \times U_1)$.
 
\section{Main Results} \label{S3}
In this section, we present our main results -- where we introduce a hierarchical optimization framework, which takes into account model uncertainty, under which the {\it follower} is required to respond optimally to the strategy of the {\it leader} so as to achieve the overall objectives.

\subsection{On the optimality control system for the follower} \label{S3.1}
In this subsection, we assume that the {\it follower} uses a single approximate model about the system (i.e., the representation in \eqref{Eq12}). Suppose that, for a given {\it leader} strategy $u_1 \in L^2((0,\,T) \times U_1)$, if $u_2^{\ast} \in L^2((0,\,T) \times U_2)$, i.e., the strategy for the {\it follower}, is an optimal solution to \eqref{Eq13} (cf. \eqref{Eq10}). Then, such a solution is characterized by the following optimality condition\footnote{Notice that, in \eqref{Eq10} (see also \eqref{Eq16} and \eqref{Eq19}), the {\it follower}'s observation subspace (which is associated with $y(t, x; (u_1, \digamma(u_1)))$) is the whole subspace $L^2((0,\,T) \times \Omega)$, whereas its strategy subspace is restricted to $L^2((0,\,T) \times U_2)$.}
\begin{align}
{\int\int}_{(0,\,T) \times \Omega} \bigl(y - y^{r_{f}}\bigr)\hat{y} dx dt + \beta {\int\int}_{(0,\,T) \times U_2} u_2^{\ast} \hat{u}_2 dx dt = 0, \notag\\
  \forall \hat{u}_2 \in L^2((0,\,T) \times U_2),  \label{Eq16}
\end{align}
where $y$ and $\hat{y}$ are, respectively, unique solutions to the following PDEs
\begin{eqnarray}
\left.\begin{array}{l}
\dfrac{\partial y}{\partial t} + \mathcal{L}_{t,x} y = u_1\chi_{U_1} + u_2^{\ast} \chi_{U_2} \quad \text{in} \quad (0, T) \times \Omega  \\
 y(0, x) = 0 \quad \text{on} \quad \Omega  \\
 y(t, x) = 0 \quad \text{for} \quad (t, x) \in \Sigma
\end{array}\right\}  \label{Eq17}
\end{eqnarray}
and
\begin{eqnarray}
\left.\begin{array}{l}
\dfrac{\partial \hat{y}}{\partial t} + \mathcal{L}_{t,x} \hat{y} = u_2^{\ast} \chi_{U_2} \quad \text{in} \quad (0, T) \times \Omega  \\
 \hat{y}(0, x) = 0 \quad \text{on} \quad \Omega  \\
 \hat{y}(t, x) = 0 \quad \text{for} \quad (t, x) \in \Sigma
\end{array}\right\}.   \label{Eq18}
\end{eqnarray}
Furthermore, if we introduce an adjoint state $p$ as follow
\begin{eqnarray}
\left.\begin{array}{l}
-\dfrac{\partial p}{\partial t} + \mathcal{L}_{t,x}^{\ast} p = y - y^{r_f} \quad \text{in} \quad (0, T) \times \Omega  \\
 p(T, x) = 0 \quad \text{on} \quad \Omega  \\
 p(t, x) = 0 \quad \text{for} \quad (t, x) \in \Sigma
\end{array}\right\},   \label{Eq19}
\end{eqnarray}
where $\mathcal{L}_{t,x}^{\ast}$ is the adjoint operator of $\mathcal{L}_{t,x}$.

Then, we have the following result which characterizes the mapping $\digamma$ in \eqref{Eq14} (i.e., the optimality system for the {\it follower} with respect to an approximate model about the system).
\begin{proposition}\label{P1}
Let $u_1 \in L^2((0,\,T) \times U_1)$ be given. Suppose that the following coupled PDEs
\begin{eqnarray}
\left.\begin{array}{l}
\dfrac{\partial y}{\partial t} + \mathcal{L}_{t,x} y = u_1\chi_{U_1} - \dfrac{1}{\beta} p \chi_{U_2}, \quad \text{in} \quad (0, T) \times \Omega  \\
-\dfrac{\partial p}{\partial t} + \mathcal{L}_{t,x}^{\ast} p = y - y^{r_f} \quad \text{in} \quad (0, T) \times \Omega  \\
 y(0, x) = 0 \quad \text{on} \quad \Omega  \\
 p(T, x) = 0 \quad \text{on} \quad \Omega  \\
 y(t, x) = p(t, x) = 0 \quad \text{for} \quad (t, x) \in \Sigma\\
\end{array}\right\},   \label{Eq20}
\end{eqnarray}
admits a unique solution pair $\bigl(y(u_1), p(u_1) \bigr)$ (which also depends uniformly on $u_1 \in L^2((0,\,T) \times U_1)$). Then, the optimality system for the follower is given by
\begin{align}
\digamma(u_1) &= -\dfrac{1}{\beta} p(u_1) \chi_{U_2} \notag \\
                            &\equiv u_2^{\ast}. \label{Eq21}
\end{align}
\end{proposition}

\begin{proof}
For a given $u_1 \in L^2((0,\,T) \times U_1)$, let $y$ and $p$ be the unique solutions of \eqref{Eq20}. If we multiply the second equation in \eqref{Eq20} by $\hat{y}$ and integrate by parts. Further, noting the PDEs in \eqref{Eq18} and \eqref{Eq19}, then we have the following
\begin{align}
{\int\int}_{(0,T)\times \Omega} \bigl(y - y^{r_f}\bigr) \hat{y} d xdt &= {\int\int}_{(0,T)\times \Omega} \biggl(-\dfrac{\partial p}{\partial t} + \mathcal{L}_{t,x}^{\ast} p \biggr) \hat{y} d xdt \notag \\
& = {\int\int}_{(0,T)\times \Omega} p \biggl(\dfrac{\partial \hat{y}}{\partial t} + \mathcal{L}_{t,x} \hat{y} \biggr) d xdt  \notag \\
& = {\int\int}_{(0,T)\times U_2} p u_2^{\ast} d xdt.  \label{Eq22}
\end{align}
Moreover, using the optimality condition in \eqref{Eq16} together with \eqref{Eq22}, we obtain
\begin{align}
p \chi_{U_2} + \beta u_2^{\ast} = 0, \label{Eq23}
\end{align}
which further gives an optimal strategy for the {\it follower} as
\begin{align*}
 u_2^{\ast} &= -\dfrac{1}{\beta} p \chi_{U_2} \\
  &\triangleq \digamma(u_1),
\end{align*}
where $p$ is from the unique solution set $\{p(u_1), y(u_1)\}$ of \eqref{Eq20} that depends uniformly on $u_1 \in L^2((0,\,T) \times U_1)$. This completes the proof of Proposition~\ref{P1}.
\end{proof}

\begin{remark} \label{R3}
The above proposition states that if the strategy of the {\it leader} $u_1 \in L^2((0,\,T) \times U_1)$ is given. Then, the strategy for the {\it follower} $u_2^{\ast} = \digamma(u_1)$, which is responsible for keeping the state trajectory $y(t,x; (u_1, \digamma(u_1)))$ close to the given reference trajectory $y^{r_f}(t, x)$ on the time intervals $(0, T)$, is optimal. Moreover, the uniqueness of $\digamma$ is implied by the existence of unique solution set to the coupled PDEs in \eqref{Eq20} (see also \eqref{Eq23}). Later, in Proposition~\ref{P2}, we provide an additional optimality condition on the strategy of the {\it leader}, when such a correspondence is interpreted in the context of hierarchical optimization framework with model uncertainty.   
\end{remark}

\subsection{On the optimality system for the leader with model uncertainty} \label{S3.2}
In this subsection, we provide an optimality condition on the strategy of the {\it leader} in \eqref{Eq9}, with respect to the class of alternative models about the system in \eqref{Eq7}, when the strategy for the {\it follower} satisfies the optimality condition of Proposition~\ref{P1}.

To this end, for a given $\xi \in L^2(\Omega)$, let $\varphi $ and $\vartheta$ be unique solutions to the following coupled PDEs
\begin{eqnarray}
\left.\begin{array}{l}
-\dfrac{\partial \varphi}{\partial t} + \mathcal{L}_{t,x}^{h^\ast} \varphi= \vartheta \quad \text{in} \quad (0, T) \times \Omega  \\
\dfrac{\partial \vartheta}{\partial t} + \mathcal{L}_{t,x} \vartheta = - \dfrac{1}{\beta} \varphi \chi_{U_2} \quad \text{in} \quad (0, T) \times \Omega  \\
 \vartheta(0, x) = 0 \quad \text{on} \quad \Omega  \\
 \varphi(T, x) = \xi \quad \text{on} \quad \Omega  \\
 \varphi(t, x) =  \vartheta(t, x)  = 0 \quad \text{for} \quad (t, x) \in \Sigma\\
\end{array}\right\},   \label{Eq24}
\end{eqnarray}
where $\mathcal{L}_{t,x}^{h^\ast}$ (with $\mathcal{L}_{t,x}^{h^\ast} = \mathcal{L}_{t,x}^{\ast} - a \triangledown_x \log h \cdot \triangledown_x \varphi$) is the adjoint operator of $\mathcal{L}_{t,x}^{h}$.

Next, define the following linear decompositions  
\begin{align}
y = y_0 + z \quad \text{and} \quad p = p_0 + q \label{Eq25}
\end{align}
such that $y_0$ and $p_0$ are the unique solutions to the following coupled PDEs
\begin{eqnarray}
\left.\begin{array}{l}
\dfrac{\partial y_0}{\partial t} + \mathcal{L}_{t,x} y_0 =  - \dfrac{1}{\beta} p_0 \chi_{U_2} \quad \text{in} \quad (0, T) \times \Omega  \\
-\dfrac{\partial p_0}{\partial t} + \mathcal{L}_{t,x}^{\ast} p_0 = y_0 - y^{r_f} \quad \text{in} \quad (0, T) \times \Omega  \\
 y_0(0, x) = 0 \quad \text{on} \quad \Omega  \\
 p_0(T, x) = 0 \quad \text{on} \quad \Omega  \\
 y_0(t, x) = p_0(t, x) = 0 \quad \text{for} \quad (t, x) \in \Sigma\\
\end{array}\right\}  \label{Eq26}
\end{eqnarray}
with
\begin{align}
a(t, x) \triangledown_x \log h(t, x) \cdot \triangledown_x y_0 (t, x) = 0 \quad \text{in} \quad (0, T) \times \Omega   \label{Eq27}
\end{align}
and 
\begin{align}
a(t, x) \triangledown_x \log h(t, x) \cdot \triangledown_x p_0(t, x) = 0 \quad \text{in} \quad (0, T) \times \Omega.   \label{Eq28}
\end{align}
Note that, from \eqref{Eq20}, \eqref{Eq26}--\eqref{Eq28} together with \eqref{Eq25}, it is easy to show that $z$ and $q$ are the unique solutions to the following PDE 
\begin{eqnarray}
\left.\begin{array}{l}
\dfrac{\partial z}{\partial t} + \mathcal{L}_{t,x}^{h} z = u_1^{\ast} \chi_{U_1} - \dfrac{1}{\beta} q \chi_{U_2} \quad \text{in} \quad (0, T) \times \Omega  \\
-\dfrac{\partial q}{\partial t} + \mathcal{L}_{t,x}^{h^\ast} q = z \quad \text{in} \quad (0, T) \times \Omega  \\
 z(0, x) = 0 \quad \text{on} \quad \Omega  \\
 q(T, x) = 0 \quad \text{on} \quad \Omega  \\
 z(t, x) = q(t, x)  = 0 \quad \text{for} \quad (t, x) \in \Sigma\\
\end{array}\right\},   \label{Eq29}
\end{eqnarray}
where $u_1^{\ast} \in L^2((0,\,T) \times U_1)$ is an optimal strategy for the {\it leader} which satisfies additional conditions (see below \eqref{Eq31} and \eqref{Eq32}).

In what follows, let us denote the norm in $L^2(\Omega)$ by $\Vert \cdot \Vert_{L^2(\Omega)}$ and assume that $\xi \in L^2(\Omega)$ satisfies the following
\begin{align}
\bigl(z(T),\, \xi \bigr) = 0, \quad \forall u_1 \in L^2((0,\,T) \times U_1), \label{Eq30}
\end{align}
where $(\cdot,\, \cdot)$ denotes the scalar product in $L^2(\Omega)$. 

Then, we have the following result which characterizes the optimality condition for the {\it leader} in \eqref{Eq9}, with respect to the class of alternative models about the system in \eqref{Eq7}.
\begin{proposition} \label{P2}
The optimal strategy for the leader that minimizes 
\begin{align*}
& \inf_{u_1 \in L^2((0,\,T) \times U_1)}  J_1(u_1)\\
& \quad \text{ s.t.} \quad  y(T; (u_1, \digamma(u_1))) \in y^{t_{g}} + \alpha B 
\end{align*}
is given by
\begin{align}
u_1^{\ast} = \varphi (\xi) \chi_{U_1}, \label{Eq31}
\end{align}
where $\varphi(\xi)$ is given from the unique solution set $\bigl\{y(\xi), p(\xi), \varphi(\xi), \vartheta(\xi) \bigr\}$ for the optimality system with model uncertainty
\begin{eqnarray}
\left.\begin{array}{l}
\dfrac{\partial y}{\partial t} + \mathcal{L}_{t,x}^{h} y = u_1^{\ast} \chi_{U_1} - \dfrac{1}{\beta} p \chi_{U_2} \quad \text{in} \quad (0, T) \times \Omega  \\
-\dfrac{\partial \varphi}{\partial t} + \mathcal{L}_{t,x}^{h^\ast} \varphi = \vartheta \quad \text{in} \quad (0, T) \times \Omega  \\
\dfrac{\partial \vartheta}{\partial t} + \mathcal{L}_{t,x} \vartheta = - \dfrac{1}{\beta} \varphi \chi_{U_2} \quad \text{in} \quad (0, T) \times \Omega  \\
-\dfrac{\partial p}{\partial t} + \mathcal{L}_{t,x}^{\ast} p = y - y^{r_f} \quad \text{in} \quad (0, T) \times \Omega  \\
 y(0, x) = \vartheta(0, x) = 0 \quad \text{on} \quad \Omega  \\
 p(T, x) = 0 \quad \text{on} \quad \Omega  \\
 \varphi(T, x) = \xi \quad \text{on} \quad \Omega  \\
 y(t, x) = p(t, x)=  \varphi(t, x) =  \vartheta(t, x)  = 0 \quad \text{for} \quad (t, x) \in \Sigma\\
\end{array}\right\}.   \label{Eq32}
\end{eqnarray}
Moreover, $\xi \in L^2(\Omega)$ is a unique solution to the following variational inequality\footnote{Notice that, in \eqref{Eq33}, we write $y(T;\,\xi)$ to make explicitly the fact that the solution set $\bigl\{y(\xi), p(\xi), \varphi(\xi), \vartheta(\xi) \bigr\}$ of \eqref{Eq32} depends uniformly on $\xi \in L^2(\Omega)$.}
\begin{align}
\bigl(y(T; \xi) -y^{t_g},\, \hat{\xi} - \xi\bigr) + \alpha \bigl(\Vert \hat{\xi}\Vert_{L^2(\Omega)} - \Vert \xi \Vert_{L^2(\Omega)} \bigr) \ge 0, \quad \forall \hat{\xi} \in L^2(\Omega). \label{Eq33}
\end{align}
\end{proposition}

\begin{proof}
Note that the optimization problem for the {\it leader} in \eqref{Eq9} is equivalent to
\begin{align*}
  &\inf_{u_1} \dfrac{1}{2}{\int\int}_{(0,\,T) \times U_1}   u_1^2 dx dt  \\
            & \quad \text{ s.t.} \quad  y(T; (u_1, \digamma(u_1))) \in y^{t_{g}} - y_0(T) + \alpha B,
\end{align*}
with respect to the linear decompositions in \eqref{Eq25}.
 
Next, introduce the following cost functionals
\begin{align}
\bar{J}_1(u_1)  = \dfrac{1}{2}{\int\int}_{(0,\,T) \times U_1} u_1^2 dx dt \label{Eq34}
\end{align}
and
\begin{align}
\bar{J}_2(u_1) = \biggl\{\begin{array}{l}
0 \quad\quad  \text{if} \quad \xi \in y^{t_{g}} - y_0(T) + \alpha B  \\
+\infty \quad \text{otherwise on} \,\, L^2(\Omega)
\end{array} \label{Eq35}
\end{align}
Let $\mathcal{H} \in \mathscr{L}(L^2((0,\,T) \times U_1); L^2(\Omega))$ be a bounded linear operator such that\footnote{$\mathscr{L}(L^2((0,\,T) \times U_1); L^2(\Omega))$ denotes a family of bounded linear operators.}
\begin{align}
\mathcal{H} u_1 = z(T; u_1). \label{Eq36}
\end{align}
Then, the optimization problem in \eqref{Eq9} is equivalent to
\begin{align}
  \inf_{u_1 \in L^2((0,\,T) \times U_1)} \biggl\{ \bar{J}_1(u_1)  + \bar{J}_2(u_1) \biggr\}. \label{Eq37}
\end{align}
Furthermore, using Fenchel's duality theorem (e.g., see \cite{Roc67} or \cite{EkeT76}), we have the following
\begin{align}
  \inf_{u_1 \in L^2((0,\,T) \times U_1)} \biggl\{ \bar{J}_1(u_1)  + \bar{J}_2(u_1) \biggr\} = - \inf_{\xi \in L^2(\Omega)} \biggl\{ \bar{J}_1^{\ast}(\mathcal{H}^{\ast} \xi)  + \bar{J}_2^{\ast}(-\xi) \biggr\}, \label{Eq38}
\end{align}
where $\mathcal{H}^{\ast}$ is the adjoint operator of $\mathcal{H}$ and the conjugate functions $\bar{J}_i^{\ast}$ are  given by
\begin{align}
  \bar{J}_i^{\ast} (\varphi) = \sup_{\hat{\varphi}} \biggl\{ (\varphi, \hat{\varphi}) - \bar{J}_i(\hat{\varphi}) \biggr\}, \quad i = 1, 2. \label{Eq39}
\end{align}
Note that if we multiply the first equation (respectively, the second one) in \eqref{Eq32} by $z$ (respectively, by $q$) and integrate by parts, then we obtain the following
\begin{align}
 (z(T), \xi) = {\int\int}_{(0,\,T) \times U_1} \varphi u_1^{\ast} dx dt.  \label{Eq40}
\end{align}
Then, for $\xi \in L^2(\Omega)$ that satisfies \eqref{Eq30}, we have the following 
\begin{align}
  \mathcal{H}^{\ast} \xi = \varphi \chi_{U_1}, \label{Eq41}
\end{align}
where $\varphi$ is from the unique solutions of \eqref{Eq32}. 

Note that
\begin{align}
  \bar{J}_1^{\ast} (u_1^{\ast}) = J_1(u_1^{\ast}) \label{Eq42}
\end{align}
and
\begin{align}
  \bar{J}_2^{\ast} (\xi) = (\xi, y^{t_{g}} - y_0(T)) + \alpha \Vert \xi \Vert_{L^2(\Omega)}. \label{Eq43}
\end{align}
Then, the optimization problem in \eqref{Eq9} is equivalent to 
\begin{align}
 & \inf_{\xi} \dfrac{1}{2}{\int\int}_{(0,\,T) \times U_1} \varphi^2 dx dt + \alpha \Vert \xi \Vert_{L^2(\Omega)} - (\xi, y^{t_{g}} - y_0(T)) \notag \\
  & \quad \text{ s.t.} \quad  y(T; (u_1, \digamma(u_1))) \in y^{t_{g}} - y_0(T) + \alpha B. \label{Eq44}
\end{align}
Let $\xi \in L^2(\Omega)$ be a unique solution to the following variational inequality
\begin{align}
{\int\int}_{(0,\,T) \times U_1} \varphi(\hat{\varphi} - \varphi) dx dt  + \bigl(y(T; \xi) -y^{t_g},\, \hat{\xi} - \xi\bigr) +& \alpha \bigl(\Vert \hat{\xi}\Vert_{L^2(\Omega)} - \Vert \xi \Vert_{L^2(\Omega)} \bigr) \ge 0, \notag \\
&\quad \quad  \forall \hat{\xi} \in L^2(\Omega). \label{Eq45}
\end{align}
Moreover, if we multiply the first equation (respectively, the second one) in \eqref{Eq24} by $(\hat{\varphi} - \varphi)$ (respectively, by $(\hat{\vartheta} - \vartheta))$ and integrate by parts, we obtain the following
\begin{align}
 {\int\int}_{(0,\,T) \times U_1} \varphi(\hat{\varphi} - \varphi) dx dt = \bigl(z(T), \hat{\xi} - f\bigr). \label{Eq46}
\end{align}
Thus, if we substitute \eqref{Eq46} into \eqref{Eq45}, then we obtain \eqref{Eq33}. This completes the proof of Proposition~\ref{P2}.
\end{proof}

\begin{remark} \label{R4}
Note that the hierarchical control problem in Proposition~\ref{P2} requires the {\it follower} to respond optimally to the strategy of the {\it leader}, where such a correspondence is implicitly embedded in \eqref{Eq32}. Such a hierarchical framework takes, at the optimization stage, into account the issue of model uncertainty in terms of constraints on the entire path of perturbed solutions, as both the {\it leader} and the {\it follower} choose their strategies based on different models about the system.
\end{remark}

\subsection{On the controllability-type problem with model uncertainty} \label{S3.3}
In the following, we further consider the controllability-type problem in \eqref{Eq15}, where we provide a condition under which $y(T; (u_1^{\ast}, \digamma(u_1^{\ast})))$ spans a dense subset of $L^2(\Omega)$, when $u_1^{\ast}$ spans the subspace $L^2((0,\,T) \times U_1)$.

\begin{proposition} \label{P3}
Suppose that Proposition~\ref{P2} holds, then, for every $y^{t_g} \in L^2(\Omega)$ and $\alpha > 0$ (which is arbitrary small), there exits $u_1^{\ast} \in L^2((0,\,T) \times U_1)$ such that
\begin{align}
y(T; (u_1^{\ast}, \digamma(u_1^{\ast}))) \in y^{t_{g}} + \alpha B. \label{Eq47}
\end{align}
\end{proposition}

\begin{proof}
From Proposition~\ref{P2}, suppose that the PDE in \eqref{Eq32} admits unique solutions (i.e., $y(t,x; \xi)$, $p(t,x; \xi)$, $\varphi(t,x; \xi)$ and $\vartheta(t,x; \xi)$ for $(t, x)\in (0,\,T) \times \Omega$ that depend uniformly on $\xi \in L^2(\Omega)$). Then, noting \eqref{Eq30}, the condition in \eqref{Eq40} becomes
\begin{align}
 \varphi(t,x) = 0 \quad \text{for} \quad (t, x) \in (0,\,T) \times U_1,  \label{Eq48}
\end{align}
which implies the following conditions (see also \eqref{Eq24})
\begin{align*}
 \vartheta \chi_{U_1} &= 0, \\
\dfrac{\partial \vartheta}{\partial t} + \mathcal{L}_{t,x} \vartheta &= 0 \quad \text{in} \quad (0,\,T) \times \bigl(U \setminus U_2\bigr)
\end{align*}
and
\begin{align*}
-\dfrac{\partial \varphi}{\partial t} + \mathcal{L}_{t,x}^{\ast} \varphi = 0 \quad \text{in} \quad (0,\,T) \times \bigl(U \setminus U_2\bigr).
\end{align*}
Furthermore, using Mizohata's uniqueness theorem (e.g., see \cite{SauSc87}) together with the regularity conditions on $\mu(t, x)$ and $a(t, x)$, then we obtain the following
\begin{align}
 \vartheta (t,x) = 0  \quad \text{for} \quad (t, x) \in (0,\,T) \times \bigl(U \setminus U_2\bigr), \label{Eq49}
\end{align}
which requires $\xi$ to have a zero value outside of $U_2$ (cf. \eqref{Eq41} and \eqref{Eq30}). 

Next, consider the restriction of $\varphi$ and $\vartheta$ to $(0,\,T) \times U_2$ such that
\begin{eqnarray}
\left.\begin{array}{l}
-\dfrac{\partial \varphi}{\partial t} + \mathcal{L}_{t,x}^{h^\ast} \varphi = \vartheta \quad \text{in} \quad (0, T) \times U_2  \\
\dfrac{\partial \vartheta}{\partial t} + \mathcal{L}_{t,x} \vartheta = - \dfrac{1}{\beta} \varphi \chi_{U_2} \quad \text{in} \quad (0, T) \times U_2  \\
 \varphi, \,\,  \dfrac{\partial \varphi}{\partial x^i}, \,\, \vartheta, \,\, \dfrac{\partial \vartheta}{\partial x^i} = 0, \quad \text{for} \quad (t, x) \in (0, T) \times \partial \, U_2 \\
 \varphi(T, x) = \xi \chi_{U_1}, \quad \vartheta(0, x) = 0  \quad \text{on} \quad U_2 \\
\end{array}\right\}. \label{Eq50}
\end{eqnarray}
Then, from \eqref{Eq50}, it remains to show that $\xi \chi_{U_2} = 0$, which is a sufficient condition for $y(T; (u_1^{\ast}, \digamma(u_1^{\ast})))$ to span a dense subset of $L^2(\Omega)$, when $u_1^{\ast}$ spans the subspace $L^2((0,\,T) \times U_1)$. 

Noting the relations in \eqref{Eq27} and \eqref{Eq28} (together with \eqref{Eq26} and \eqref{Eq29}), the first two equations in \eqref{Eq50} imply the following
\begin{align}
\left(\dfrac{\partial}{\partial t} + \mathcal{L}_{t,x}^{h} \right) \left(-\dfrac{\partial}{\partial t} + \mathcal{L}_{t,x}^{h^{\ast}} \right) \varphi(t,x) + \dfrac{1}{\beta} \varphi(t, x) = 0 \quad \text{in} \quad (0, T) \times U_2, \label{Eq51}
\end{align}
which is a quasi-elliptic equation; and in view of Cauchy problems on bounded domains (e.g., see \cite[Theorem~6.6.1]{StrVa79}), for any fixed $t \in (0, T)$, $\varphi(t, x)$ is analytic in $U_2$, with Cauchy data zero on smooth boundary $\partial \,U_2$ of class $C^2$. As a result of this, $\varphi(t, x) = 0$ on $(0, T) \times \partial \,U_2$ and also continuous in $t$, then we have $\varphi(0, x) = 0$ and $\varphi(t, x)=\vartheta(t, x)=0$ for $(t, x) \in (0, T) \times \partial \, U_2$, which implies  $\xi \chi_{U_2} = 0$ (cf. the PDE in \eqref{Eq50}, since $\varphi(T, x) = \xi \chi_{U_1}$). This completes the proof of Proposition~\ref{P3}.
\end{proof}

\end{document}